\documentclass[11pt]{article}
\usepackage{graphicx}
\usepackage{amscd, amssymb}
\usepackage{enumerate}
\usepackage{hyperref}
\usepackage{lscape}
\newenvironment{eq}{\begin{equation}}{\end{equation}}
\newenvironment{proof}{{\bf Proof}:}{\vskip 5mm }

\newtheorem{proposition}{Proposition}[subsection]
\newtheorem{lemma}[proposition]{Lemma}
\newtheorem{definition}[proposition]{Definition}

\newtheorem{example}[proposition]{Example}
\newtheorem{remark}[proposition]{Remark}

\newtheorem{problem}[proposition]{Problem}
\newtheorem{construction}[proposition]{Construction}
\newcommand{\llabel}[1]{\label{#1}}
\newcommand{\comment}[1]{}
\newcommand{\sr}{\rightarrow}

\newcommand{\zz}{{\bf Z\rm}}

\newcommand{\nn}{{\bf N\rm}}
\newcommand{\nat}{\nn}

\newcommand{\mbind}{\rho}

{\vskip
3mm}

\begin{document}
\parskip = 2mm
\begin{center}
{\bf\Large Lawvere theories and Jf-relative monads}


\vspace{3mm}

{\large\bf Vladimir Voevodsky}\footnote{School of Mathematics, Institute for Advanced Study,
Princeton NJ, USA. e-mail: vladimir@ias.edu}
\vspace {3mm}

{\large\bf December 2015 - January 2016}  
\end{center}

\begin{abstract}
In this paper we provide a detailed construction of an equivalence between the category of Lawvere theories and the category of relative monads on the obvious functor $Jf:F\sr Sets$ where $F$ is the category with the set of objects $\nat$ and morphisms being the functions between the standard finite sets of the corresponding cardinalities. The methods of this paper are fully constructive and it should be formalizable in the Zermelo-Fraenkel theory without the axiom of choice and the excluded middle. It is also easily formalizable in the UniMath. 
\end{abstract}

\tableofcontents

\subsection{Introduction}
The notion of a relative monad is introduced in \cite[Def.1, p. 299]{ACU} and considered in more detail in \cite{ACU2}. The categories of relative monads are parametrized by functors rather than by categories, i.e., while one speaks of a monad on a category $C$ one speaks of a relative monad on a functor $J:C\sr D$. We reminds the relevant definitions and constructions of \cite{ACU2} in the first section of the paper. 

Following \cite{FPT} we let $F$ denote the category with the set of objects $\nat$ and the sets of morphisms $Mor_F(m,n)$ being the sets of functions $stn(m)\sr stn(n)$ where $stn(n)=\{i\in\nat\,|\,i<n\}$ is the standard set with $n$ elements. 

For a universe $U$ let $Sets(U)$ be the category of sets in $U$ (see a detailed definition in Section \ref{LRML}). For any $U$ there is an obvious functor $Jf_U:F\sr Sets(U)$. The main construction of the paper is a construction of an equivalence between the category $RMon(Jf_U)$ of relative monads on $Jf_U$ and the category $LW(U)$ of Lawvere theories in $U$ (see \cite{LandC} for the precise definition of $LW(U)$). 

While the main idea of this construction is straightforward its detailed presentation requires a considerable amount of work. In particular, since we work, as in \cite{LandC}, in the Zermelo-Fraenkel set theory without the axiom of choice and without the excluded middle axiom, we had to reprove a number of results about coproducts. One of the unexpected discoveries was the fact that it is impossible to construct the finite coproducts structure on the category $F$ and that instead one has to work with a weaker structure of finite ordered coproducts.  

We use the diagrammatic order in writing compositions, i.e., for $f:X\sr Y$ and $g:Y\sr Z$ we write $f\circ g$ for the composition of $f$ and $g$. 

We do not make precise the concept of a {\em universe} that we use for some of the statements of the paper. It would be certainly sufficient to assume that $U$ is a Grothendieck universe. However, it seems likely that sets $U$ satisfying much weaker conditions can be used both for the statements and for the proofs of our results.  

The problem/construction pairs in the paper can be interpreted in the ZF-formalization as follows. The ``problem'' part is formalized as a formula $P(x_1,\dots,x_n)$ with the free variables $x_1,\dots,x_n$ corresponding to the objects introduced in the problem. The ``construction part'' is formalized as a theorem of the form ``there exist unique $x_1,\dots,x_n$ such that $P(x_1,\dots,x_n)$ and $Q(x_1,\dots,x_n)$'' where $Q$ is a formula expressing the detailed properties of the objects defined by the construction. For example, formulas $P$ and $Q$ in the ZF-formalization of the problem ``to construct a homomorphism of groups $H:G_1\sr G_2$'' with the construction ``Let $G_1=\zz/2$, $G_2=\zz/2$ and $H=Id_{\zz/2}$'' will be as follows. The formula $P(G_1,G_2,H)$ will be expressing the fact that $G_1$ is a group, $G_2$ is a group and $H$ is a homomorphism from $G_1$ to $G_2$. The formula $Q(G_1,G_2,H)$ will be expressing the fact that $G_1=\zz/2$, $G_2=\zz/2$ and $H$ equals the identity homomorphism of $\zz/2$. One can envision a proof assistant with the user-level language being some convenient dependently typed language that translates this language into formulas and deductions of the ZF and then verifies these formulas and deductions according to the rules of the first-order logic.

\subsection{Relative monads}

\begin{definition}
\llabel{2015.12.22.def1}
Let $J:C\sr D$ be a functor. A relative monad ${\bf RR}$ on $J$ or a $J$-relative monad is a collection of data of the form
\begin{enumerate}
\item a function $RR:Ob(C)\sr Ob(D)$,
\item for each $X$ in $C$ a morphism $\eta(X):J(X)\sr RR(X)$,
\item for each  $X,Y$ in $C$ and $f:J(X)\sr RR(Y)$ a morphism $\mbind(f):RR(X)\sr RR(Y)$,
\end{enumerate}
such that the following conditions hold:
\begin{enumerate}
\item for any $X\in C$, $\mbind(\eta(X))=Id_{RR(X)}$,
\item for any $f:J(X)\sr RR(Y)$, $\eta(X)\circ \mbind(f)=f$,
\item for any $f:J(X)\sr RR(Y)$, $g:J(Y)\sr RR(Z)$, 
$$\mbind(f)\circ \mbind(g)=\mbind(f\circ \mbind(g))$$
\end{enumerate}
\end{definition}
The following definition repeats \cite[Definition 2.2, p.4]{ACU2}.
\begin{definition}\llabel{2015.12.22.def2}
Let $J:C\sr D$ be a functor and ${\bf RR}=(RR,\eta,\mbind)$, ${\bf RR}'=(RR',\eta',\mbind')$ be two relative monads on $J$. A morphism $\phi:{\bf RR}\sr {\bf RR}'$ is a function $\phi:Ob(C)\sr Mor(D)$ that to each $X\in C$ assigns a morphism $\phi(X):RR(X)\sr RR'(X)$ such that 
\begin{enumerate}
\item for any $X\in C$ one has $\eta'(X)=\eta(X)\circ \phi(X)$,
\item for any $f:J(X)\sr RR(Y)$ one has 
$$\mbind(f)\circ \phi(Y)=\phi(X)\circ \mbind'(f\circ \phi(Y))$$
\end{enumerate}
\end{definition}
\begin{lemma}
\llabel{2015.12.22.l1}
Let $J:C\sr D$ be a functor and ${\bf RR}$ a relative monad on $J$. Then the function $X\mapsto Id_{RR(X)}$ is a morphism of relative monads ${\bf RR}\sr {\bf RR}$.
\end{lemma}
\begin{proof}
Both conditions of Definition \ref{2015.12.22.def2} are straightforward to prove.
\end{proof}
\begin{lemma}
\llabel{2015.12.22.l2}
Let $J:C\sr D$ be a functor and ${\bf RR},{\bf RR}',{\bf RR}''$ be relative monads on $J$. Then if $\phi$ and $\phi'$ are functions $Ob(C)\sr Mor(D)$ which are morphisms of relative monads ${\bf RR}\sr {\bf RR}'$ and ${\bf RR}'\sr {\bf RR}''$ then the function $X\mapsto \phi(X)\circ \phi'(X)$ is a morphism ${\bf RR}\sr {\bf RR}''$. 
\end{lemma}
\begin{proof}
Let $X\in C$ then
$$\eta(X)\circ \phi(X)\circ \phi'(X)=\eta'(X)\circ \phi'(X)=\eta''(X)$$
this proves the first condition of Definition \ref{2015.12.22.def2}. To prove the second condition let $f:J(X)\sr RR(Y)$ then we have
$$\mbind(f)\circ \phi(Y)\circ \phi'(Y)=\phi(X)\circ \mbind(f\circ \phi(Y))\circ \phi'(Y)=\phi(X)\circ \phi'(X)\circ \mbind(f\circ \phi(Y)\circ \phi'(Y))$$
\end{proof}
\begin{problem}\llabel{2015.12.22.prob3}
Let $J:C\sr D$ be a functor. To construct a category $RMon(J)$ of relative monads on $J$.
\end{problem}
\begin{construction}\rm\llabel{2015.12.18.constr3}
Applying the same approach as before we obtain category data with the set of objects being the set $RMon(J)$ of relative monads on $J$, the set of morphisms 
being the set of triples $(({\bf RR},{\bf RR}'),\phi)$ where ${\bf RR}$, ${\bf RR}'$ are relative monads on $J$ and $\phi$ is a morphism of relative monads from ${\bf RR}$ to ${\bf RR}'$ as 
given by Definition \ref{2015.12.22.def2}, the identity morphisms are given by Lemma \ref{2015.12.22.l1} and compositions by Lemma \ref{2015.12.22.l2}. It follows immediately from the corresponding properties of morphisms in $C$ that these data satisfies the left and right identity and the associativity axioms forming a category. The set of morphisms from ${\bf RR}$ to ${\bf RR}'$ in this category is not equal to the set of morphisms of relative monads but it is in the obvious bijective correspondence with this set and we will use both functions of this bijective correspondence as coercions\footnote{When a function $f:X\sr Y$ is declared as a {\em coercion} then every time that one has an expression $a$ that denotes an element of the set $X$ in a position where an element of the set $Y$ is expected one replaces it by $f(a)$}.  
\end{construction}  
\begin{lemma}
\llabel{2016.01.03.l5}
Let $\phi:{\bf RR}\sr {\bf RR}'$ be a morphism of relative monads on $J:C\sr D$ such that for all $X\in C$ the morphism $\phi(X):RR(X)\sr RR'(X)$ is an isomorphism. Then $\phi$ is an isomorphism in the category of relative monads on $J$.
\end{lemma}
\begin{proof}
Set $\phi'(X)=(\phi(X))^{-1}$. In view of the definition of the composition of morphisms of relative monads and the identity morphism of relative monads it is sufficient to verify that the family $\phi'$ is a morphism of relative monads from ${\bf RR}'$ to ${\bf RR}$. That it is the inverse to $\phi$ is then straightforward to prove. 

Let us check the two conditions of Definition \ref{2015.12.22.def2}. The equality 
$$\eta(X)=\eta'(X)\circ \phi'(X)$$
follows from the equality $\eta'(X)=\eta(X)\circ \phi(X)$ by composing it with $\phi'(X)$ on the right and using the fact that $\phi(X)\circ \phi'(X)=Id_{RR(X)}$.

The second condition is of the form, for any $f':J(X)\sr RR'(Y)$,
\begin{eq}\llabel{2016.01.03.eq2}
\mbind'(f')\circ\phi'(Y)=\phi'(X)\circ\mbind'(f'\circ\phi'(Y))
\end{eq}
Applying the second condition of Definition \ref{2015.12.22.def2} for $\phi$ to $f=f'\circ \phi'(Y)$ and using the equality $\phi'(Y)\circ\phi(Y)=Id_{RR'(Y)}$ we get
$$\mbind(f'\circ\phi'(Y))\circ\phi(Y)=\phi(X)\circ\mbind'(f'\circ \phi'(Y)\circ\phi(Y))=\phi(X)\circ\mbind'(f')$$
It remains to compose this equality with $\phi'(Y)$ on the right and $\phi'(X)$ on the left and rewrite the equalities $\phi(Y)\circ\phi'(Y)=Id_{RR(Y)}$ and $\phi'(X)\circ \phi(X)=Id_{RR'(X)}$.
\end{proof}

Let us remind the definition of the Kleisli category of a relative monad (see \cite[p.8]{ACU2}). 
\begin{problem}\llabel{2015.12.22.prob1}
Let $J:C\sr D$ be a functor and ${\bf RR}$ be a relative monad on $J$. To define a category $K({\bf RR})$ that will be called Kleisli category of ${\bf RR}$.
\end{problem}
\begin{construction}\rm\llabel{2015.12.22.constr3}
We set $Ob(K({\bf RR}))=Ob(C)$ and 
$$Mor(K({\bf RR}))=\amalg_{X,Y\in Ob(K({\bf RR}))}Mor(J(X),RR(Y))$$
We will, as before, identify the set of morphisms in $K({\bf RR})$ from $X$ to $Y$ with $Mor(J(X),RR(Y))$ by means of the obvious bijections. 

For $X\in Ob(C)$ we set $Id_{X,K({\bf RR})}=\eta(X)$.

For $f\in Mor(J(X),RR(Y))$, $g\in Mor(J(Y),RR(Z))$ we set $f\circ_{K({\bf RR})}g=f\circ_D \mbind(g)$.

Verification of the associativity and the left and right identity axioms of a category are straightforward.
\end{construction}
\begin{problem}\llabel{2015.12.22.prob2}
Let $J:C\sr D$ be a functor and ${\bf RR}$ be a relative monad on $J$. To construct a functor $L_{{\bf RR}}:C\sr K({\bf RR})$.
\end{problem}
\begin{construction}\rm\llabel{2015.12.22.constr4}
We set $L_{Ob}=Id$ and for $f:X\sr Y$, $L(f)=J(f)\circ_D\eta(Y)$. Verification of the identity and composition axioms of a functor are straightforward.
\end{construction}
The following lemma will be needed below.
\begin{lemma}
\llabel{2016.01.03.l4b}
Let $u:X\sr Y$ in $C$ and $g:J(Y)\sr RR(Z)$ in $D$. Then one has
$$L_{{\bf RR}}(u)\circ_{K({\bf RR})} g=J(u)\circ_D g$$
\end{lemma}
\begin{proof}
One has
$$L_{{\bf RR}}(u)\circ_{K({\bf RR})} g=L_{{\bf RR}}(u)\circ_D\mbind(g)=J(u)\circ_D\eta(Y)\circ_D\mbind(g)=J(u)\circ_D g$$
\end{proof}
\begin{problem}\llabel{2015.12.22.prob4}
Let $J:C\sr D$ be a functor and $\phi:{\bf RR}\sr {\bf RR}'$ a morphism of relative monads on $J$. To construct a functor $K(\phi):K({\bf RR})\sr K({\bf RR}')$ such that $L_{{\bf RR}}\circ K(\phi)=L_{{\bf RR}'}$.
\end{problem}
\begin{construction}\rm\llabel{2015.12.22.constr5}
This construction is not, as far as we can tell, described in \cite{ACU2} and we will do all computations in detail.

We set $K(\phi)_{Ob}=Id$. For $f\in Mor_D(J(X),RR(Y))$ we set 
$$K(\phi)(f)=f\circ_D \phi(Y).$$
For the identity axiom of a functor we have
$$K(\phi)(Id_{X,K({\bf RR})})=K(\phi)(\eta_X)=\eta_X\circ_D \phi(X)=\eta'_X=Id_{X,K({\bf RR}')}$$
For the composition axiom, for $f\in Mor_D(J(X),RR(Y))$, $g\in Mor_D(J(Y),RR(Z))$ we have
$$K(\phi)(f\circ_{{\bf RR}} g)=K(\phi)(f\circ_D \mbind(g))=f\circ_D \mbind(g)\circ_D \phi(Z)=f\circ_D \phi(Y)\circ_D\mbind'(g\circ_D\phi(Z))$$
and
$$K(\phi)(f)\circ_{{\bf RR}'} K(\phi)(g)=(f\circ_D\phi(Y))\circ_{{\bf RR}'}(g\circ_D\phi(Z))=f\circ_D\phi(Y)\circ_D\mbind'(g\circ_D\phi(Z))$$
The condition $L_{{\bf RR}}\circ K(\phi)=L_{{\bf RR}'}$ obviously holds on objects and on morphisms we have for $f\in Mor_C(X,Y)$: 
$$(L_{{\bf RR}}\circ K(\phi))(f)=K(\phi)(L_{{\bf RR}}(f))=K(\phi)(J(f)\circ_D\eta(Y))=J(f)\circ_D\eta(Y)\circ_D\phi(Y)=$$$$J(f)\circ_D\eta'(Y)=L_{{\bf RR}'}(f).$$
Construction \ref{2015.12.22.constr5} is completed. 
\end{construction}
\begin{lemma}
\llabel{2016.01.01.l2}
Let $J:C\sr D$ be a functor. Then one has:
\begin{enumerate}
\item for a relative monad ${\bf RR}$ on $J$, $K(Id_{{\bf RR}})=Id_{K({\bf RR})}$,
\item for morphisms $\phi:{\bf RR}\sr {\bf RR}'$, $\phi':{\bf RR}'\sr {\bf RR}''$ of relative monads on $J$, $K(\phi\circ \phi')=K(\phi)\circ K(\phi')$.
\end{enumerate}
\end{lemma}
\begin{proof}
The first assertion follows from the right identity axiom for $D$.

The second assertion follows from the associativity of composition in $D$.
\end{proof}

\subsection{Binary coproducts and finite ordered coproducts in the constructive setting} 

In the absence of Axiom of Choice (AC) the structure of finite coproducts on a category can not be obtained from an initial object and the structure of binary coproducts. The same, of course, is true for products - the proof of \cite[Prop.1, p. 73]{MacLane} essentially depends on the AC. However, binary coproducts allow one to construct finite {\em ordered} coproducts as described below. 
\begin{definition}\llabel{2015.12.20.def1}
A binary coproducts structure on a category $C$ is a function that assigns to any pair of objects $X,Y$ of $C$ an object $X\amalg Y$ and two morphisms 
$$ii_0^{X,Y}:X\sr X\amalg Y$$
$$ii_1^{X,Y}:Y\sr X\amalg Y$$
such that for any object $W$ of $C$ and any two morphisms $f_X:X\sr W$, $f_Y:Y\sr W$ there exists a unique morphism $\Sigma(f_X,f_Y):X\amalg Y\sr W$ such that
$$ii_0^{X,Y}\circ\Sigma(f_X,f_Y)=f_X$$
$$ii_1^{X,Y}\circ\Sigma(f_X,f_Y)=f_Y$$
\end{definition}
\begin{definition}
\llabel{2015.12.24.def2}
A finite ordered coproduct structure on a category $C$ is a function that for any $m\ge 0$ and any sequence $X=(X_0,\dots,X_{m-1})$ of objects of $C$ defines an object $\amalg_{i=0}^{m-1}X_i$ and morphisms $ii^X_i:X_i\sr \amalg_{i=0}^{m-1}X_i$ such that for any sequence $f_i:X_i\sr Y$, $i=0,\dots,m-1$ there exists a unique morphism $\Sigma_{i=0}^{m-1} f_i:\amalg_{i=0}^{m-1}X_i\sr Y$ such that 
\begin{eq}\llabel{2015.12.24.eq2}
ii^X_j\circ \Sigma_{i=0}^{m-1} f_i = f_j
\end{eq}
\end{definition}
Note that for $m=0$ there is a unique sequence of the form $(X_0,\dots,X_{m-1})$ - the empty sequence, and the corresponding $\amalg_{i=0}^{m-1}X_i$ is an initial object of $C$. 
\begin{problem}\llabel{2015.12.24.prob1}
Given a category $C$ with an initial object $0$ and a binary coproducts structure to construct a finite ordered coproducts structure on $C$.
\end{problem}
\begin{construction}\rm\llabel{2015.12.24.constr1}
By induction on $m$.

For $m=0$ one defines $\amalg X_i$ to be $0$. The construction of the morphism $\Sigma f_i$, in this case for the empty set of morphisms $f_i$, and its properties follow easily from the definition of an initial object.

For $m=1$ one defines $\amalg X_i=X_0$, $ii^X_0=Id_{X_0}$ and $\Sigma f_i=f_0$. The verification of the conditions is again straightforward.

For the successor one defines 
$$\amalg_{i=0}^m X_i=(\amalg_{i=0}^{m-1}X_i)\amalg X_m$$
and
$$\Sigma_{i=0}^m f_i=\Sigma(\Sigma_{i=0}^{m-1}f_i, f_m)$$
The morphisms $ii^X_i$ for $i=0,\dots,m-1$ are given by
$$ii^X_i=ii^{X'}_i\circ ii^{\amalg_{i=0}^{m-1}X_i , X_m}_0$$
where $X'$ is the sequence $(X_0,\dots,X_{m-1})$, and
$$ii^X_m=ii^{\amalg_{i=0}^{m-1}X_i , X_m}_1$$
To show that $\Sigma_{i=0}^{m}f_i$ satisfies the condition of Definition \ref{2015.12.24.constr1} we have:
\begin{enumerate}
\item for $j<m$
$$ii^X_j\circ \Sigma_{i=0}^{m} f_i = ii^{X}_j\circ \Sigma(\Sigma_{i=0}^{m-1}f_i, f_m)
=ii^{X'}_j\circ ii^{\amalg_{i=0}^{m-1}X_i , X_m}_0 \circ \Sigma(\Sigma_{i=0}^{m-1}f_i, f_m)=ii^{X'}_j\circ \Sigma_{i=0}^{m-1}f_i=f_j$$
where the third equation is from the definition of a binary coproduct,
\item for $j=m$
$$ii^X_m\circ \Sigma_{i=0}^{m} f_i=ii^{\amalg_{i=0}^{m-1}X_i , X_m}_1 \Sigma(\Sigma_{i=0}^{m-1}f_i, f_m)=f_m$$
\end{enumerate}
To show that $f=\Sigma_{i=0}^mf_i$ is a unique morphism satisfying these conditions let $g$ be another morphism such that 
$$ii^X_j\circ g=f_j$$
for all $j=0,\dots,m$. Both $f$ and $g$ are morphisms from $(\amalg_{i=0}^{m-1}X_i)\amalg X_m$. By the uniqueness condition of Definition \ref{2015.12.20.def1} 
it is sufficient to show that 
$$ii^{\amalg_{i=0}^{m-1}X_i , X_m}_0\circ f=ii^{\amalg_{i=0}^{m-1}X_i , X_m}_0\circ g$$
and
$$ii^{\amalg_{i=0}^{m-1}X_i , X_m}_1\circ f=ii^{\amalg_{i=0}^{m-1}X_i , X_m}_1\circ g$$
To prove the first equality it is sufficient, by the inductive assumption, to prove that
$$ii^{X'}_j\circ ii^{\amalg_{i=0}^{m-1}X_i , X_m}_0\circ f=ii^{X'}_j\circ ii^{\amalg_{i=0}^{m-1}X_i , X_m}_0\circ g$$
for all $j=0,\dots,m-1$. This follows from our assumption since
$$ii^{X'}_j\circ ii^{\amalg_{i=0}^{m-1}X_i , X_m}_0=ii^X_j$$
Similarly, the second equality follows from our assumption because
$$ii^{\amalg_{i=0}^{m-1}X_i , X_m}_1=ii^X_m.$$

This completes Construction \ref{2015.12.24.constr1}. 
\end{construction}
\begin{lemma}
\llabel{2016.01.03.l4}
Let $C$ be a category with an initial object $0$ and binary coproducts structure $(\amalg, ii_0, ii_1)$. Let $(\amalg',ii'_i)$ be the finite ordered coproducts structure defined on $C$ by Construction \ref{2015.12.24.constr1}. Then for $X=(X_0,X_1)$ one has
$$(\amalg')_{i=0}^1X_i=X_0\amalg X_1$$
and
$$(ii')_0^X=ii_0^{X_0,X_1}$$
$$(ii')_1^X=ii_1^{X_0,X_1}$$
\end{lemma}
\begin{proof}
The proof is by unfolding Construction \ref{2015.12.24.constr1} in the case $m=2$.
\end{proof}
\begin{lemma}
\llabel{2015.12.24.l5}
Given a category $C$ with the finite ordered  coproducts structure $(\amalg_i X_i,ii^X_i)$ let $f_i:X_i\sr Y$ where $i=0,\dots,m-1$ and $g:Y\sr Z$. Then one has
\begin{eq}\llabel{2015.12.24.eq4}
(\Sigma_i f_i)\circ g=\Sigma_i ( f_i\circ g)
\end{eq}
\end{lemma}
\begin{proof}
By the uniqueness condition of Definition \ref{2015.12.24.def2} it is sufficient to show that for all $i=0,\dots, m-1$ the precompositions of both sides of (\ref{2015.12.24.eq4}) with $ii^X_i$ are equal. We have
$$ii^X_i\circ (\Sigma_i f_i)\circ g=f_i\circ g=ii^X_i\circ (f_i\circ g)$$
\end{proof}
\begin{lemma}
\llabel{2015.12.24.l4}
Let $C$ be a category with a finite ordered coproducts structure and $(X_0,\dots,X_{m-1})$ a sequence of objects of $C$. Then one has
$$\Sigma_{i=0}^{m-1} ii_i^X=Id_{\amalg_{i=0}^{m-1}X_i}$$
\end{lemma}
\begin{proof}
It follows from the uniqueness part of Definition \ref{2015.12.24.def2}.
\end{proof}

\begin{definition}\llabel{2016.01.01.def1}
Let $(C,\amalg,ii_0,ii_1)$ and $(C',\amalg',ii_0',ii_1')$ be two categories with the binary coproducts structure. A functor $G:C\sr C'$ is said to strictly respect the binary coproduct structures if for all $X,Y\in C$ one has:
$$G(X\amalg Y)=G(X)\amalg' G(Y)$$
and
$$G(ii_0^{X,Y})=(ii_0')^{X,Y}$$
$$G(ii_1^{X,Y})=(ii_1')^{X,Y}$$
\end{definition}
\begin{definition}\llabel{2016.01.01.def2}
Let $(C,\amalg,ii_i)$ and $(C',\amalg',ii'_i)$ be two categories with finite ordered coproducts structures. A functor $G:C\sr C'$ is said to strictly respect the finite ordered coproducts structures if for all $n\in\nat$ and all sequences $X=(X_0,\dots,X_{m-1})$ one has
$$G(\amalg_{i=0}^mX_i)=(\amalg')_{i=0}^{m-1}G(X_i)$$
and for all $i=0,\dots,m-1$ one has
$$G(ii_i^X)=(ii')_i^{G(X)}$$
\end{definition}
\begin{lemma}
\llabel{2016.01.01.l3}
Let $(C,\amalg,ii_0,ii_1)$ and $(C',\amalg',ii_0',ii_1')$ be two categories with the binary coproducts structure and let $0$, $0'$ be initial objects in $C$ and $C'$ respectively. Let $G:C\sr C'$ be a functor. Then $G$ strictly respects the finite coproduct structure on $C$ and $C'$ defined by the initial object and the binary coproduct structure by Construction \ref{2015.12.24.constr1} if and only if one has:
\begin{enumerate}
\item $G(0)=0'$,
\item $G$ strictly respects the binary coproduct structure.
\end{enumerate}
\end{lemma}
\begin{proof}
The "only if" part follows from the fact that the initial objects of $C$ and $C'$  defined by the finite ordered coproducts structure of Construction \ref{2015.12.24.constr1} are $0$ and $0'$ and Lemma \ref{2016.01.03.l4}.

The proof of the "if" part is easy by induction on the length of the sequence $X=(X_0,\dots,X_m)$ of Definition \ref{2016.01.01.def2}. 
\end{proof}
\begin{remark}\rm\llabel{2016.01.05.rem1}
It is not true in general that a finite ordered coproducts structure is determined by the corresponding initial object and the binary coproducts structure. In particular, the converse of Lemma \ref{2016.01.01.l3} is false - a functor that strictly respects the initial object and the binary coproducts structure defined by a finite ordered coproducts structure need not strictly respect the finite ordered coproducts structure itself.
\end{remark}
\begin{lemma}\llabel{2016.01.01.l6}
Let $(C,\amalg,ii_i)$ and $(C',\amalg',ii'_i)$ be two categories with finite ordered coproducts structures and $G:C\sr C'$ a functor that strictly respect the finite ordered coproducts structures.

Let $X=(X_0,\dots,X_{m-1})$ be a sequence of objects of $C$ and $f_i:X_i\sr Y$ a sequence of morphisms. Then one has
\begin{eq}\llabel{2016.01.01.eq2}
G(\Sigma_{i=0}^{m-1} f_i)=\Sigma_{i=0}^{m-1} G(f_i)
\end{eq}
where the $\Sigma$ on the left is with respect to $(\amalg,ii_i)$ and $\Sigma$ on the right is with respect to $(\amalg',ii'_i)$.
\end{lemma}
\begin{proof}
Both the left and the right hand side of (\ref{2016.01.01.eq2}) are morphisms from $\amalg_{i=0}^{m-1}G(X_i)$ to $G(Y)$ according to the Definition \ref{2016.01.01.def2}. The right hand side is the unique morphism with these domain and codomain such that for all $i=0,\dots,m-1$ its pre-composition with $(ii')_i^{G(X)}$ equals $G(f_i)$. It remains to show that the same property holds for the right hand side. We have
$$(ii')_i^{G(X)}\circ G(\Sigma_{i=0}^{m-1} f_i)=G(ii_i^X)\circ G(\Sigma_{i=0}^{m-1} f_i)=G(ii_i^X\circ \Sigma_{i=0}^{m-1} f_i)=G(f_i)$$.
The lemma is proved. 
\end{proof}

\subsection{More on the category $F$}

Following \cite{FPT} we let $F$ denote the category with the set of objects $\nat$ and the set of morphisms from $m$ to $n$ being $Fun(stn(m),stn(n))$, where $stn(m)=\{i\in\nat\,|\,i<m\}$ is our choice for the standard set with $m$ elements (cf. \cite{LandC}).

For $m,n\in\nat$ let $ii_0^{m,n}:stn(m)\sr stn(m+n)$ and $ii_1^{m,n}:stn(n)\sr stn(m+n)$ be the injections of the initial segment of length $m$ and the concluding segment of length $n$.
\begin{lemma}
\llabel{2016.01.03.l2}
One has:
\begin{enumerate}
\item $0$ is the initial object of $F$,
\item the function 
$$(m,n)\mapsto (m+n, ii_0^{m,n}, ii_1^{m,n})$$
is a binary coproduct structure on $F$.
\end{enumerate}
\end{lemma}
\begin{proof}
We have $stn(0)=\emptyset$ and there is a unique function from $\emptyset$ to any other set.

The second assertion can be reduced to the case $n=1$ by induction on $n$ and then proved by direct reasoning involving the details of the set-theoretic definition of a function.
\end{proof}
\begin{definition}
\llabel{2016.01.03.d1}
The binary coproducts structure on $F$ defined by Lemma \ref{2016.01.03.l2} is called the standard binary coproducts structure. 

The finite ordered coproducts structure on $F$ defined by Lemma \ref{2016.01.03.l2} and Construction \ref{2015.12.24.constr1} is called the standard finite ordered coproducts structure.
\end{definition}
\begin{example}
\llabel{2016.01.03.ex1}\rm
There are binary coproducts structures on $F$ that are different from the standard binary coproducts structure. For example, the function that is equal to the standard binary coproducts structure on all pairs $(m,n)$ other than $(1,1)$ and such that $1\amalg 1=2$, $ii_0^{1,1}(0)=1$ and $ii_1^{1,1}=0$ is a binary coproducts structure on $F$ that is not equal to the standard one.
\end{example}
\begin{remark}\rm
\llabel{2016.01.03.rem1} 
It is easy to define the concept of a finite coproducts structure on a category. The only non-trivial choice one has to make is which of the definitions of a finite set to use and it is reasonable to define a finite set as a set for which there exists, in the ordinary logical sense, $m\in\nat$ and a bijection from $stn(m)$ to this set.

One can show then that it is impossible to construct a finite coproducts structure on $F$ without using the axiom of choice. Indeed, one would have to define for each finite set $I$ and a function $X:I\sr \nat$ the coproduct object $\amalg X=\amalg_{i\in I}X(i)\in\nat$ and a family of functions
$$ii_i^X:stn(X(i))\sr stn(\amalg X)$$
for $i\in I$ such that for any $n$ the function
$$Fun(stn(\amalg X), stn(n))\sr \prod_{i\in I}Fun(stn(X(i)),stn(\amalg X))$$
defined by this family is a bijection. The latter condition is easily shown to be equivalent to the condition that
$$stn(\amalg X)=\amalg_{i\in I}Im(ii_i^X)$$
One can also prove that if such a structure exists then $\amalg X=\Sigma_{i\in I} X(i)$ where the sum on the right is the usual commutative sum in $\nat$. Consider the case when $I$ is a set with $2$ elements and $X(i)=1$ for all $i\in I$. Then $\amalg X = 2$ and $ii_i^X:stn(1)\sr stn(2)$ are functions whose images do not intersect and cover $stn(2)$. Then the function $i\mapsto ii_i^X(0)$ is a bijection from $I$ to $stn(2)$, i.e., we have found a canonical bijection from any finite set with 2 elements to $stn(2)$. This amounts to a particular case of the axiom of choice for the proper class of all sets with $2$ elements or, if we consider finite coproducts relative to a universe $U$, for the set of sets with $2$ elements in $U$.  
\end{remark}
\begin{lemma}
\llabel{2016.01.03.l3}
Consider $F$ with the standard finite ordered coproducts structure. Then for any $m\in\nat$, $n_0,\dots,n_{m-1}\in\nat$ one has:
\begin{enumerate}
\item$\amalg_{i=0}^{m-1}n_i=\Sigma_{i=0}^{m-1}n_i$,
\item for each $i=0,\dots, m-1$ and $j=0,\dots, k_i-1$ one has
$$ii_i^{(n_0,\dots,n_{m-1})}(j)=(\Sigma_{l=0}^{i-1}n_l)+j$$
In particular, $ii_i^{(1,\dots,1)}(0)=i$.
\end{enumerate}
\end{lemma}
\begin{proof}
By induction on $m$ using Construction \ref{2015.12.24.constr1}.
\end{proof}

\subsection{Lawvere theories}
Lawvere theories were introduced in \cite{Lawvere}. Let us remind an equivalent but more direct definition here.
\begin{definition}
\llabel{2015.11.24.def1}
A Lawvere theory structure on a category $T$ is a functor $L:F\sr T$ such that the following conditions hold:
\begin{enumerate}
\item $L$ is a bijection on the sets of objects,
\item $L(0)$ is an initial object of $T$,
\item for any $m,n\in\nat$ the square
$$
\begin{CD}
L(0) @>>> L(n)\\
@VVV @VVL(ii_1^{m,n}) V\\
L(m) @>L(ii_0^{m,n})>> L(m+n)
\end{CD}
$$
is a push-out square.
\end{enumerate}
A Lawvere theory is a pair $(T,L)$ where $T$ is a category and $L$ is a Lawvere theory structure on $T$. 
\end{definition}

\begin{lemma}
\llabel{2015.12.24.l3}
A functor $L:F\sr T$ is a Lawvere structure on $T$ if an only if it is bijective on objects, $L(0)$ is an initial object of $T$ and the function 
$$(X,Y)\sr (L(L^{-1}(X)+L^{-1}(Y)), L(ii_0^{L^{-1}(X),L^{-1}(Y)}), L(ii_1^{L^{-1}(X),L^{-1}(Y)}))$$
is a binary coproducts structure on $T$.
\end{lemma}
\begin{proof}
It follows by unfolding definitions and rewriting the equalities $L(L^{-1}(X))=X$ and $L^{-1}(L(n))=n$.
\end{proof}
\begin{definition}
\llabel{2016.01.03.def2}
Let $(T,L)$ be a Lawvere theory. The binary coproducts structure on $T$ defined in Lemma \ref{2015.12.24.l3} is called the standard binary coproducts structure defined by (the Lawvere theory structure) $L$. 

The finite ordered coproducts structure on $T$ defined by the initial object $L(0)$ and the standard binary coproducts structure on $T$ by Construction \ref{2015.12.24.constr1} is called the standard finite ordered coproducts structure defined by $L$.
\end{definition}
Everywhere below, unless the opposite is explicitly stated, we consider, for a Lawvere theory $(T,L)$ the category $T$ with the standard binary coproduct and finite ordered coproduct structures.
\begin{lemma}\llabel{2015.12.24.l5b}
Let $(T,L)$ be a Lawvere theory. Then $L$ strictly respects the standard finite coproduct structures on $F$ and $T$, i.e., for any $m\in\nat$, $n_0,\dots,n_{m-1}\in \nat$ one has:
\begin{enumerate}
\item $\amalg_{i=0}^{m-1} L(n_i)=L(\Sigma_{i=0}^{m-1} n_i)$,
\item for any $i=0,\dots,m-1$, 
$$L(ii_i^{(n_0,\dots,n_{m-1})})=ii_i^{(L(n_0),\dots,L(n_{m-1}))}$$
\end{enumerate}
\end{lemma}
\begin{proof}
Simple by induction on $m$ using the explicit form of Construction \ref{2015.12.24.constr1}.
\end{proof}
\begin{lemma}\llabel{2016.01.05.l1}
Let $(T,L)$ be a Lawvere theory and let $u\in Fun(stn(m),stn(n))$. Then one has
$$L(u)=\Sigma_{i=0}^{m-1}ii^{(L(1),\dots,L(1))}_{u(i)}$$
\end{lemma}
\begin{proof}
Both sides of the equality are morphisms from $L(m)$ to $L(n)$ in $T$. Since by Lemma \ref{2015.12.24.l5b}(1) $L(m)$ is the finite coproduct of the sequence $(L(1),\dots,L(1))$ to prove that two morphisms from $L(m)$ are equal it is sufficient to prove that their pre-compositions with $ii^{(L(1),\dots,L(1))}_i$ are equal for all $i=0,\dots,m-1$. We have
$$ii^{(L(1),\dots,L(1))}_i\circ \Sigma_{i=0}^{m-1}ii^{(L(1),\dots,L(1))}_{u(i)}=ii^{(L(1),\dots,L(1))}_{u(i)}=L(ii^{(1,\dots,1)}_{u(i)})$$
and
$$ii^{(L(1),\dots,L(1))}_i\circ L(u)=L(ii^{(1,\dots,1)}_i)\circ L(u)=L(ii^{(1,\dots,1)}_i\circ u)$$
It remains to show that
$$ii^{(1,\dots,1)}_{u(i)}=ii^{(1,\dots,1)}_i\circ u$$
in $F$. Since both sides are functions from $stn(1)$ it is sufficient to prove that their values on $0$ are equal. This follows from Lemma \ref{2016.01.03.l3}. 
\end{proof}

Recall that a morphism of Lawvere theories $G:(T,L)\sr (T',L')$ is a functor $G:T\sr T'$ such that $L\circ G=L'$.
\begin{lemma}
\llabel{2015.01.01.l4}
Let $G:(T,L)\sr (T',L')$ be a morphism of Lawvere theories. Then $G$ strictly respects the binary coproduct structures of Lemma \ref{2015.12.24.l3}. 
\end{lemma}
\begin{proof}
It follows by unfolding definitions and rewriting the equalities $L(L^{-1}(X))=X$ and $L^{-1}(L(n))=n$.
\end{proof}
\begin{lemma}
\llabel{2016.01.01.l5}
Let $G:(T,L)\sr (T',L')$ be a morphism of Lawvere theories. Then $G$ strictly respects the standard ordered finite coproduct structures on $T$ and $T'$.
\end{lemma}
\begin{proof}
It follows directly from Lemmas \ref{2016.01.01.l3} and \ref{2015.01.01.l4} and the equality $G(L(0))=(L\circ G)(0)=L'(0)$.
\end{proof}

\subsection{Lawvere theories and $Jf$-relative monads}
\llabel{LRML}

Let us start by reminding that for any set $U$ there is a category $Sets(U)$ of the following form. The set of objects of $Sets(U)$ is $U$. The set of morphisms is
$$Mor(Sets(U))=\cup_{X,Y\in U}Fun(X,Y)$$
Since a function from $X$ to $Y$ is defined as a triple $(X,Y,G)$ where $G$ is the graph subset of this function the domain and codomain functions are well defined on $Mor(Sets(U))$ such that
$$Mor_{Sets(U)}(X,Y)=Fun(X,Y)$$
and a composition function can be defined that restricts to the composition of functions function on each $Mor_{Sets(U)}(X,Y)$. Finally the identity function $U\sr Mor(Sets(U))$ is obvious and the collection of data that one obtains satisfies the axioms of a category. This category is called the category of sets in $U$ and denoted $Sets(U)$. 

We will only consider the case when $U$ is a universe.

Following \cite{ACU} we let $Jf_U:F\sr Sets(U)$ denote the functor that takes $n$ to $stn(n)$ and that is the identity on morphisms between two objects (on the total sets of morphisms the morphism component of this functor is the inclusion of a subset). Recall that we use the expression ``a $J$-relative monad'' as a synonym for the expression ``a relative monad on $J$''. 

By simply unfolding definitions we get the following explicit form for the definition of a $Jf_U$-relative monad. 
\begin{lemma}
\llabel{2016.01.01.l1}
A $Jf_U$-relative monad is a collection of data of the form:
\begin{enumerate}
\item for each $n\in\nat$ a set $RR(n)$ in $U$,
\item for each $n\in\nat$ a function $stn(n)\sr RR(n)$,
\item for each $m,n\in \nat$ and $f:stn(m)\sr RR(n)$, a function $\mbind(f):RR(m)\sr RR(n)$,
\end{enumerate}
such that the following conditions hold:
\begin{enumerate}
\item for all $n\in\nat$, $\mbind(\eta(n))=Id_{RR(n)}$,
\item for all $f:stn(m)\sr RR(n)$, $\eta(m)\circ \mbind(f)=f$,
\item for all $f:stn(k)\sr RR(m)$, $g:stn(m)\sr RR(n)$, $\mbind(f)\circ \mbind(g)=\mbind(f\circ \mbind(g))$.
\end{enumerate}
\end{lemma}
The main goal of this section is to provide a construction for the following problem.
\begin{problem}\llabel{2016.01.05.prob1}
For a universe $U$ to construct an equivalence between the category $LW(U)$ of Lawvere theories in $U$ and the category $RMon(Jf_U)$ of $Jf_U$-relative monads.
\end{problem}
The construction will be given in Construction \ref{2016.01.05.constr1} below. 
\begin{lemma}\llabel{2015.12.22.l3}
Let ${\bf RR}$ be a relative monad on $Jf:F\sr Sets(U)$. Then $(K({\bf RR}),L_{{\bf RR}})$ is a Lawvere theory.
\end{lemma}
\begin{proof}
We need to prove that the pair $(K({\bf RR}),L_{{\bf RR}})$ satisfies conditions of Definition \ref{2015.11.24.def1}. The first condition is obvious. The second condition is also obvious since $Fun(stn(0),RR(n))$ is a one point set for any set $RR(n)$. The third condition is straightforward to prove as well since the square 
$$
\begin{CD}
Fun(stn(m+n),RR(k)) @>ii_1^{m,n}\circ\_>> Fun(stn(n),RR(k))\\
@Vii_0^{m,n}\circ\_VV @VVV\\
Fun(stn(m),RR(k)) @>>> Fun(stn(0),RR(k))
\end{CD}
$$
is a pull-back square for any set $RR(k)$.
\end{proof}
\begin{problem}\llabel{2016.01.01.prob1}
To construct a functor $RML_U:RMon(Jf_U)(U)\sr LW(U)$.
\end{problem}
\begin{construction}\rm\llabel{2015.12.22.def5}
We define the object component of $RML$ setting
$$RML_{Ob}({\bf RR})=(K({\bf RR}),L_{{\bf RR}})$$
It is well defined by Lemma \ref{2015.12.22.l3}. 

We define the morphism component of $RLM$ setting
$RML_{Mor}(\phi)=K(\phi)$.
It is well defined by the condition of Problem \ref{2015.12.22.prob4}. 

The identity and composition axioms of a functor follow from Lemma \ref{2016.01.01.l2}. 
\end{construction}
Below we consider, for a Lawvere theory $(T,L)$, the category $T$ with the finite ordered coproducts structure obtained by applying Lemma \ref{2015.12.24.l3} and Construction \ref{2015.12.24.constr1}. 
\begin{problem}\llabel{2015.12.22.prob5}
Let $U$ be a universe and $(T,L)$ a Lawvere theory in $U$. To construct a $Jf_U$-relative monad ${\bf RR}=(RR,\eta,\mbind)$.
\end{problem}
\begin{construction}\rm\llabel{2015.12.22.constr6}
We set:
\begin{enumerate}
\item $RR(n)=Mor_T(L(1),L(n))$,
\item $\eta(n)$ is the function $stn(n)\sr Mor_T(L(1),L(n))$ given by
$$\eta(n)(i) = ii_i^{(L(1),\dots,L(1))}.$$ 
This function is well defined because 
$$\amalg_{i=0}^{n-1}L(1)=L(n)$$
by Lemma \ref{2015.12.24.l5b},
\item for $f\in Fun(stn(m),Mor_T(L(1),L(n)))$ we define 
$$\mbind(f)\in Fun(Mor_T(L(1),L(m)), Mor_T(L(1),L(n)))$$
as $g\mapsto g\circ \Sigma_{i=0}^{m-1} f(i)$. This formula is again well-defined in view of Lemma \ref{2015.12.24.l5b}.
\end{enumerate}
Let us verify the conditions of Lemma \ref{2016.01.01.l1}. 

For the first condition we have 
$$\mbind(\eta(n))(g)=g\circ \Sigma_{i=0}^{n-1}\eta(n)(i)=g\circ \Sigma_{i=0}^{n-1}ii_i^{(L(1),\dots,L(1))}=g\circ Id_{L(n)}=g$$
where the third equality is by Lemma \ref{2015.12.24.l4}.

For the second condition let $f\in Fun(stn(m),Mor_T(L(1),L(n)))$. To verify that $\eta(m)\circ \mbind(f)=f$ we need to verify that these two functions from $stn(m)$ are equal, i.e., that for each $i=0,\dots,m-1$ we have
$$(\eta(m)\circ \mbind(f))(i)=f(i)$$
We have
$$(\eta(m)\circ \mbind(f))(i)=\mbind(f)(\eta(m)(i))=\mbind(f)(ii_i^{(L(1),\dots,L(1))})=ii_i^{(L(1),\dots,L(1))}\circ \Sigma_{j=0}^{m-1}f(j)=f(i)$$

To prove the third condition we need to show that
$$\mbind(f)\circ \mbind(g)=\mbind(f\circ \mbind(g))$$
for all  $f\in Fun(stn(k),Mor_T(L(1),L(m)))$ and $g\in Fun(stn(m),Mor_T(L(1),L(n)))$. 

Both sides are functions from $Mor_T(L(1),L(k))$. To verify that they are equal we need to show that for any $h\in Mor_T(L(1),L(k))$ we have
$$(\mbind(f)\circ \mbind(g))(h)=\mbind(f\circ \mbind(g))(h)$$
We have
$$(\mbind(f)\circ \mbind(g))(h)=\mbind(g)(\mbind(f)(h))=\mbind(g)(h\circ \Sigma_{i=0}^{k-1}f(i))=h\circ (\Sigma_{i=0}^{k-1} f(i))\circ (\Sigma_{j=0}^{m-1} g(j))$$
and
$$\mbind(f\circ \mbind(g))(h)=h\circ (\Sigma_{i=0}^{k-1} (f\circ \mbind(g))(i))=h\circ (\Sigma_{i=0}^{k-1} (\mbind(g)(f(i))))=h\circ (\Sigma_{i=0}^{k-1} (f(i)\circ \Sigma_{j=0}^{m-1}g(j)))$$
The right hand sides of these two expressions are equal by Lemma \ref{2015.12.24.l5}. This completes the construction.
\end{construction}
We let $LRM(T,L)$ denote the $Jf_U$-relative monad defined in Construction \ref{2015.12.22.constr6}.
\begin{problem}\llabel{2016.01.01.prob2}
Let $G:(T,L)\sr (T',L')$ be a morphism of Lawvere theories. To construct a morphism of relative monads $LRM(T,L)\sr LRM(T',L')$.
\end{problem}
\begin{construction}\rm\llabel{2016.01.01.constr2}
We need to construct a family of functions 
$$\phi(n):Mor_{T}(L(1),L(n))\sr Mor_{T'}(L'(1),L'(n))$$
that satisfies the conditions of Definition \ref{2015.12.22.def2} for $J=Jf$ and relative monads $LRM(T,L)=(RR,\eta,\mbind)$ and $LRM(T',L')=(RR',\eta',\mbind')$. Set
$$\phi(n)=G_{L(1), L(n)}$$
since $L'=L\circ G$ these functions have the correct domain and codomain. 

For the first condition of Definition \ref{2015.12.22.def2} we need to show that for any $n\in\nat$ one has
$$\eta'(n)=\eta(n)\circ G_{L(1), L(n)}$$
Since both sides are functions from $stn(n)$ it is sufficient to show that for all $i=0,\dots,n-1$ one has $\eta'(n)(i)=(\eta(n)\circ G_{L(1), L(n)})(i)$. By construction
$$(\eta(n)\circ G_{L(1),L(n)})(i)=G(\eta(n)(I))=G(ii^X_i)$$
and
$$\eta'(n)(i)=ii^{X'}_i$$
where $X=(L(1),\dots,L(1))$ and $X'=(L'(1),\dots,L'(1))$. Therefore we need to show that $G(ii^X_i)=ii^{X'}_i$. This follows from Lemma \ref{2016.01.01.l5}.

For the second condition of Definition \ref{2015.12.22.def2} let $f:stn(m)\sr Mor_T(L(1),L(n))$. We need to show that
$$\mbind(f)\circ \phi(n)=\phi(m)\circ \mbind(f\circ \phi(n))$$
Both sides are functions from $Mor_T(L(1),L(m))$ to $Mor_{T'}(L'(1),L'(n))$. To show that they are equal we have to show that for each $g\in Mor_T(L(1),L(m))$ one has
$$(\mbind(f)\circ \phi(n))(g)=(\phi(m)\circ \mbind'(f\circ \phi(n)))(g)$$
For the left hand side of this equality we have:
$$(\mbind(f)\circ \phi(n))(g)=\phi(n)(\mbind(f)(g))=\phi(n)(g\circ \Sigma_{i=0}^{m-1}f(i))=G(g\circ \Sigma_{i=0}^{m-1}f(i))=G(g)\circ G(\Sigma_{i=0}^{m-1}f(i))=$$$$G(g)\circ \Sigma_{i=0}^{m-1}G(f(i))$$
where the last equality follows from Lemma \ref{2016.01.01.l6}.

For the right hand side we have:
$$(\phi(m)\circ \mbind'(f\circ \phi(n)))(g)=\mbind'(f\circ \phi(n))(\phi(m)(g))=\mbind'(f\circ \phi(n))(G(g))=G(g)\circ \Sigma_{i=0}^{m-1}(f\circ \phi(n))(i)=$$$$G(g)\circ \Sigma_{i=0}^{m-1}(\phi(n)(f(i)))=G(g)\circ \Sigma_{i=0}^{m-1}G(f(i))$$
This completes the proof of the second condition of Definition \ref{2015.12.22.def2} and the construction.
\end{construction}
We let $LRM(\phi)$ or $LRM_{Mor}(\phi)$ denote the morphism of relative monads defined by Construction \ref{2016.01.01.constr2}
\begin{problem}
\llabel{2016.01.01.prob3}
For a universe $U$, to construct a functor
$$LRM_U:LW(U)\sr RMon(Jf_U)$$
\end{problem}
\begin{construction}\rm\llabel{2016.01.01.constr5}
We define the object component of $LRM$ as the function defined by Construction \ref{2015.12.22.constr6} and the morphism component as the function defined by Construction \ref{2016.01.01.constr2}.

We need to verify that these two functions satisfy the identity and composition axioms of a functor.

Both follow immediately from the definitions of the identity functor and composition of functors. 
\end{construction}
\begin{problem}
\llabel{2016.01.01.prob4}
For any universe $U$ to construct an  isomorphism of functors
$$RML_U\circ LRM_U\sr Id_{RMon(Jf_U)}.$$
\end{problem}
\begin{construction}\rm\llabel{2016.01.01.constr6}
Let ${\bf RR}=(RR,\eta,\mbind)$ be a $Jf_U$-relative monad. Let 
$$(T,L)=RML_U(RR,\eta,\mbind)$$
and
$$(RR',\eta',\mbind')=LRM_U(T,L).$$
We need to construct an  isomorphism of relative monads
$$\phi_{{\bf RR}}:(RR',\eta',\mbind')\sr (RR,\eta,\mbind)$$
and show that the family $\phi_{{\bf RR}}$ satisfies the naturality axiom of the definition of functor morphism. 

We have
$$RR'(n)=Mor_T(L(1),L(n))=Mor_{K({\bf RR})}(L_{{\bf RR}}(1),L_{{\bf RR}}(n))=Mor_{K({\bf RR})}(1,n)=$$$$Fun(stn(1),RR(n))$$
and we define $\phi_{{\bf RR}}(n):RR'(n)\sr RR(n)$ as the obvious bijection given by setting 
$$\phi_{{\bf RR}}(n)(f)=f(0)$$
Let us show that these functions form a morphism of relative monads, i.e., that they satisfy two conditions of Definition \ref{2015.12.22.def2}. We should exchange places between the $\eta$ and $\eta'$ since we consider a morphism ${\bf RR}'\sr {\bf RR}$. The first condition becomes
$$\eta(n)(i)=(\eta'(n)\circ\phi_{{\bf RR}}(n))(i)$$
for any $n\in \nat$ and $i=0,\dots,n-1$ and the second
$$(\mbind'(f)\circ \phi_{{\bf RR}}(n))(g)=(\phi_{{\bf RR}}(m)\circ \mbind(f\circ \phi_{{\bf RR}}(n)))(g)$$
for any $f\in Fun(stn(m), RR'(n))$ and $g\in RR'(m)$.

For $n\in\nat$ and $i=0,\dots,n-1$ we have
$$(\eta'(n)\circ \phi_{{\bf RR}}(n))(i)=\phi_{\bf RR}(n)(\eta'(n)(i))=\phi_{{\bf RR}}(ii_i^{(L(1),\dots,L(1))})=ii_i^{(L(1),\dots,L(1))}(0)=$$$$L(ii_i^{(1,\dots,1)})(0)=L_{{\bf RR}}(ii_i^{(1,\dots,1)})(0)=(ii_i^{(1,\dots,1)}\circ \eta(n))(0)=\eta(n)(ii_i^{(1,\dots,1)}(0))=\eta(n)(i)$$
where the fourth equality is by Lemma \ref{2015.12.24.l5b} and the eighth equality is by Lemma \ref{2016.01.03.l3}. 

For the second condition, $f\in Fun(stn(m), RR'(n))$ and $g\in RR'(m)$ we have
$$(\mbind'(f)\circ \phi_{{\bf RR}}(n))(g)=\phi_{{\bf RR}}(n)(\mbind'(f)(g))=\phi_{{\bf RR}}(n)(g\circ_T \Sigma_{i=0}^{m-1}f(i))=(g\circ_T \Sigma_{i=0}^{m-1}f(i))(0)$$
where $f$ is considered as an element of $Fun(stn(m),Mor_T(L(1),L(n)))$ and $g$ as an element of $Mor_T(L(1),L(m))$. Next we have:
$$(g\circ_T \Sigma_{T,i=0}^{m-1}f(i))(0)=(g\circ_{K({\bf RR})} \Sigma_{T,i=0}^{m-1}f(i))(0)=(g\circ \mbind(\Sigma_{T,i=0}^{m-1}f(i)))(0)=\mbind(\Sigma_{T,i=0}^{m-1}f(i))(g(0))$$
where on the right $g$ is considered as an element of $Fun(stn(1),RR(m))$. 

On the other hand we have:
$$(\phi_{{\bf RR}}(m)\circ \mbind(f\circ \phi_{{\bf RR}}(n)))(g)=\mbind(f\circ \phi_{{\bf RR}}(n))(\phi_{{\bf RR}}(m)(g))=\mbind(f\circ \phi_{{\bf RR}}(n))(g(0))$$
where on the right $g$ is considered as an element of $Fun(stn(1),RR(m))$. 

Let us show that 
$$\Sigma_{T,i=0}^{m-1}f(i)=f\circ \phi_{{\bf RR}}(n),$$
Since both sides are morphisms in $T$ from $L(m)$ to $L(n)$ and it is sufficient to show that for any $j=0,\dots,m$ one has
$$ii_j^{(L(1),\dots,L(1))}\circ_T (\Sigma_{T,i=0}^{m-1}f(i))=ii_j^{(L(1),\dots,L(1))}\circ_T (f\circ \phi_{{\bf RR}}(n))$$
The left hand side equals $f(j)$. For the right hand side we have
$$ii_j^{(L(1),\dots,L(1))}\circ_T (f\circ \phi_{{\bf RR}}(n))=L(ii_j^{(1,\dots,1)})\circ_T (f\circ \phi_{{\bf RR}}(n))=L(ii_j^{(1,\dots,1)})\circ_{K({\bf RR})} (f\circ \phi_{{\bf RR}}(n))=$$$$
ii_j^{(1,\dots,1)}\circ f\circ \phi_{{\bf RR}}(n)$$
where the first equality is by Lemma \ref{2015.12.24.l5b} and the third equality is by Lemma \ref{2016.01.03.l4b}. Both $f(j)$ and $ii_j^{(1,\dots,1)}\circ f\circ \phi_{{\bf RR}}(n)$ are elements of $Fun(stn(1),RR(n))$. To prove that they are equal it is sufficient to prove that they coincide on $0$. We have:
$$(ii_j^{(1,\dots,1)}\circ f\circ \phi_{{\bf RR}}(n))(0)=(f\circ \phi_{{\bf RR}}(n))(i)=\phi_{{\bf RR}}(n)(f(i))=f(i)(0)$$
where the first equality is by Lemma \ref{2016.01.03.l3}(2). 

This completes the proof of the fact that the family of functions $\phi_{{\bf RR}}$ is a morphism of relative monads. 

Let us show that the family $\phi_{{\bf RR}}$ satisfies the naturality axiom of the definition of functor morphism. Let $u:{\bf RR}_1\sr {\bf RR}_2$ be a morphism of relative monads. Let $(T_i,L_i)=RML({\bf RR}_i)$ and ${\bf RR}'_i=LRM(T_i,L_i)$, $i=1,2$. Let $G=RML(u)$ and $u'=LRM(G)$. We need to show that the square
$$
\begin{CD}
{\bf RR}'_1 @>u'>> {\bf RR}'_2\\
@V\phi_{{\bf RR}_1} VV @VV\phi_{{\bf RR}_2} V\\
{\bf RR}_1 @>u>> {\bf RR}_2
\end{CD}
$$
commutes, i.e., that for any $n\in\nat$ one has
\begin{eq}\llabel{2016.01.03.eq1}
u'(n)\circ \phi_{{\bf RR}_2}(n)=\phi_{{\bf RR}_1}(n)\circ u(n)
\end{eq}
We have that 
$$u'(n)\in Fun(RR'_1(n),RR'_2(n))=Fun(Fun(stn(1),RR_1(n)),Fun(stn(1),RR_2(n)))$$
and
$$u'(n)(f)=(LRM(G)(n))(f)=G_{L_1(1),L_1(n)}(f)=G_{1,n}(f)=f\circ u(n)$$
Both sides of (\ref{2016.01.03.eq1}) are functions from $Fun(stn(1),RR_1(n))$. Therefore to prove that they are equal we need to prove that their values on any $f\in Fun(stn(1),RR_1(n))$ are equal. We have:
$$(u'(n)\circ \phi_{{\bf RR}_2}(n))(f)=\phi_{{\bf RR}_2}(n)(u'(n)(f))=(u'(n)(f))(0)=(f\circ u(n))(0)=u(n)(f(0))$$
and
$$(\phi_{{\bf RR}_1}(n)\circ u(n))(f)=u(n)(\phi_{{\bf RR}_1}(n)(f))=u(n)(f(0)).$$

This completes the proof of the fact that the family $\phi_{{\bf RR}}$ is a morphism of functors $RML_U\circ LRM_U\sr Id_{RMon(Jf_U)}$. That it is an isomorphism follows from the general properties of functor morphisms and Lemma \ref{2016.01.03.l5}. This completes Construction \ref{2016.01.01.prob4}.
\end{construction}
\begin{problem}
\llabel{2016.01.03.prob1}
For a universe $U$ to construct a functor isomorphism
$$LRM_U\circ RML_U\sr Id_{LW(U)}$$
\end{problem}
\begin{construction}\llabel{2016.01.03.constr1}\rm
Let $(T,L)$ be a Lawvere theory in $U$. Let
$$(RR,\eta,\mbind)=LRM(T,L)$$
and
$$(T',L')=RML(RR,\eta,\mbind)$$
We need to construct an isomorphism of Lawvere theories 
$$G^{(T,L)}:(T',L')\sr (T,L)$$
and show that the family $G^{(T,L)}$ is natural with respect to the morphisms of Lawvere theories $(T_1,L_1)\sr (T_2,L_2)$. While constructing $G^{(T,L)}$ we will abbreviate its notation to $G$.

We have:
$$Ob(T')=Ob(K({\bf RR}))=Ob(F)=\nat$$
$$Mor_{T'}(m,n)=Mor_{K({\bf RR})}(m,n)=Fun(stn(m),RR(n))=Fun(stn(m),Mor_T(L(1),L(n)))$$
We set the object component of $G$ to be the object component of $L$.  

We set the morphism component
$$G_{m,n}:Mor_{T'}(m,n)=Fun(stn(m),Mor_T(L(1),L(n)))\sr Mor_T(L(m),L(n))=Mor_{T}(m,n)$$
to be of the form:
$$G_{m,n}(f)=\Sigma_{T,i=0}^{m-1}f(i)$$
To show that $G_{m,n}$ is a bijection consider the function in the opposite direction given by, for $u\in Mor_T(m,n)$ and $i=0,\dots,m-1$
$$G^*_{m,n}(u)(i)=ii_i^{(L(1),\dots,L(1))}\circ u$$
The fact that $G$ and $G^*$ are mutually inverse follows easily from the definition of finite ordered coproducts. 

Let us show that $G$ is a functor. For the composition axiom, let $f\in Mor_{T'}(k,m)$, $g\in Mor_{T'}(m,n)$, then
$$G_{k,m}(f)\circ_T G_{m,n}(g)=(\Sigma_{T,i=0}^{k-1}f(i))\circ_T (\Sigma_{T,j=0}^{m-1}g(j))=\Sigma_{T,i=0}^{k-1}(f(i)\circ_T  (\Sigma_{T,j=0}^{m-1}g(j)))$$
and
$$G_{k,n}(f\circ_{T'} g)=\Sigma_{T,i=0}^{k-1}((f\circ\mbind(g))(i))=\Sigma_{T,i=0}^{k-1}(\mbind(g)(f(i)))=\Sigma_{T,i=0}(f(i)\circ_T(\Sigma_{j=0}^{m-1}g(j)))$$
where the last equality is by Construction \ref{2015.12.22.constr6}(3). 

For the identity axiom, let $n\in \nat$ then
$$G_{n,n}(Id_{T',m})=G_{n,n}(\eta(m))=\Sigma_{T,i=0}^{m-1}(\eta(m)(i))=\Sigma_{T,i=0}^{m-1}(ii_i^{(L(1),\dots,L(1))})=Id_{T,L(m)}$$
where the first equality is by Construction \ref{2015.12.22.constr3}, the third one is by Construction \ref{2015.12.22.constr6}(2) and the third one is by Lemma \ref{2015.12.24.l4}.

To prove that $G$ is a morphism of Lawvere theories we have to show that $L'\circ G=L$. On objects the equality is obvious. To show that it holds on morphisms let $u\in Fun(stn(m),stn(n))$. Then
$$(L'\circ G)(u)=G(L'(u))=\Sigma_{T,i=0}^m L'(u)(i)=\Sigma_{T,i=0}^m L_{{\bf RR}}(u)(i)=\Sigma_{T,i=0}^m (u\circ \eta(n))(i)=$$$$\Sigma_{T,i=0}^m \eta(n)(u(i))=\Sigma_{T,i=0}^m ii^{(L(1),\dots,L(1))}_{u(i)}=L(u)$$
where the fourth equality is by Construction \ref{2015.12.22.constr4} and the sixth one is by Construction \ref{2015.12.22.constr6}(2) and the seventh one is by Lemma \ref{2016.01.05.l1}.

This completes the construction of the Lawvere theory morphisms $G^{(T,L)}$. 

It remains to show that they are natural with respect to morphisms of Lawvere theories. Let $H:T_1\sr T_2$ be such a morphism. Let $(RR_i,\eta_i,\mbind_i)=LRM(T_i,L_i)$ for $i=1,2$, $(T_i',L_i')=RML(RR_i,\eta_i,\mbind_i)$, $\phi=LRM(H)$ and $H'=RML(\phi)$. 

Since $(L_i')_{Ob}=Id_{\nat}$ and $L_1'\circ H'=L_2'$ we have  that $(H')_{Ob}=Id_{\nat}$. 

For $m,n\in\nat$ and 
$$f\in Mor_{T'_1}(m,n)=Fun(stn(m),Mor_{T_1}(L_1(1),L_1(n)))$$
we have  
$$H'(f)=RML(\phi)(f)=K(\phi)(f)=f\circ \phi(n)=f\circ LRM(H)(n)=f\circ H_{L_1(1),L_1(n)}$$
where the third equality is by Construction \ref{2015.12.22.constr5} and the fifth equality is by Construction \ref{2016.01.01.constr2}.

We need to show that the square
$$
\begin{CD}
T'_1 @>H'>> T'_2\\
@VG^{(T_1,L_1)}VV @VVG^{(T_2,L_2)}V\\
T_1 @>H>> T_2
\end{CD}
$$
commutes. 

For the object components, since $(G^{(T_i,L_i)})_{Ob}=(L_i)_{Ob}$ it means that for all $n\in \nat$ one has
$$L_2(H'(n))=H(L_1(n)),$$
i.e., that $L_2(n)=H(L_1(n))$ which follows from the fact that $H$ is a morphism of Lawvere theories. 

For the morphism component it means that for all $f\in Fun(stn(m),Mor_{T_1}(L_1(1),L_1(n)))$ one has
$$G^{(T_2,L_2)}(H'(f))=H(G^{(T_1,L_1)}(f)),$$
For the left hand side we have:
$$G^{(T_2,L_2)}(H'(f))=G^{(T_2,L_2)}(f\circ H_{L_1(1),L_1(n)})=\Sigma_{T_2,i=0}^{m-1}(f\circ H_{L_1(1),L_1(n)})(i)=\Sigma_{T_2,i=0}^{m-1}(H(f(i)))$$
For the right hand side we have:
$$H(G^{(T_1,L_1)}(f))=H(\Sigma_{T_1,i=0}^{m-1}f(i))=\Sigma_{T_2,i=0}^{m-1}(H(f(i)))$$
where the second equality is by Lemmas \ref{2016.01.01.l5} and \ref{2016.01.01.l6}. 

This completes the proof that the constructed family of Lawvere theories morphisms $G^{(T,L)}$ is a morphism of functors and with it completes Construction \ref{2016.01.03.constr1}.
\end{construction}
We can now provide a construction for Problem \ref{2016.01.05.prob1}.
\begin{construction}\rm\llabel{2016.01.05.constr1}
A functor $RML_U$ from $RMon(Jf_U)$ to $LW(U)$ is provided by Construction \ref{2015.12.22.def5}. A functor $LMR_U$ from $LW(U)$ to $RMon(Jf_U)$ is provided by Construction \ref{2016.01.01.constr5}. A functor isomorphism $RML_U\circ LRM_U\sr Id_{RMon(Jf_U)}$ is provided by Construction \ref{2016.01.01.constr6}. A functor isomorphism $LRM_U\circ RML_U\sr Id_{LW(U)}$ is provided by Construction \ref{2016.01.03.constr1}.
\end{construction}
\begin{remark}\rm\llabel{2016.01.05.rem2}
The composition $RML_U\circ LRM_U$ is just slightly off from being {\em equal} to the identity functor on $RMon(Jf_U)$. It might appear that one can achieve the equality by considering a modified version $LRM'$ of the functor $LRM$ that sends $(T,L)$ to the relative monad based on the family of sets $Mor_T(L(1),L(n))^m$ where for a set $X$ and $m\in \nat$ one defines $X^m$ inductively as $X^0=stn(1)$, $X^1=X$ and $X^{n+1}=X^n\times X$. However, even this modified version of $LRM$ fails to achieve the equality due to the coercions that we need to insert to make our expression completely transparent. Indeed, the set of morphisms of the category $T$ in $(T,L)=LRM'({\bf RR})$ is $\amalg_{m,n\in\nat}RR(n)^m$ and the set $RR'(n)$ in ${\bf RR}'=RML(T,L)$ is $Mor_T(1,n)$, i.e., the set of iterated pairs of the form $((1,n),x)$ where $x\in RR(n)$. 
\end{remark}

{\em Acknowledgements:} This material is based on research sponsored by The United States Air Force Research Laboratory under agreement number FA9550-15-1-0053. The US Government is authorized to reproduce and distribute reprints for Governmental purposes notwithstanding  any copyright notation thereon.

The views and conclusions contained herein are those of the author and should not be interpreted as necessarily representing the official policies or endorsements, either expressed or implied, of the United States Air Force Research Laboratory, the U.S. Government or Carnegie Melon University.

\def\cprime{$'$}


\end{document}